\documentclass[preprintnumbers,amsmath,11pt,amssymb,floatfix,superscriptaddress,nofootinbib]{article}
\usepackage{amsmath,amssymb,latexsym,textcomp,mathrsfs}
\usepackage[all]{xy}
\usepackage{graphicx}
\usepackage{bm,amsmath, amsthm, amssymb, amsfonts}
\usepackage{times}
\usepackage[english]{babel}
\usepackage{setspace}
\setbox0=\hbox{$+$}
\newdimen\plusheight
\plusheight=\ht0
\def\+{\;\lower\plusheight\hbox{$+$}\;}

\setbox0=\hbox{$-$}
\newdimen\minusheight
\minusheight=\ht0
\def\-{\;\lower\minusheight\hbox{$-$}\;}

\setbox0=\hbox{$\cdots$}
\newdimen\cdotsheight
\cdotsheight=\plusheight%\ht0
\def\cds{\lower\cdotsheight\hbox{$\cdots$}}

\numberwithin{equation}{section}
\theoremstyle{plain}
\newtheorem{theorem}{Theorem}[section]
\newtheorem{lemma}{Lemma}[section]

\newtheorem{example}{Example}[section]
\newtheorem{definition}{Definition}[section]

\newtheorem{proposition}{Proposition}[section]

\newtheorem{note}{Note}[section]

\def\mytitle#1{\setcounter{equation}{0}
\setcounter{footnote}{0}
\begin{flushleft}
\Large\textbf{#1}
\end{flushleft}
\vspace{0.20cm}}
\def\myname#1{\leftline{#1}\vspace{-0.13cm}}
\def\myplace#1#2{\small\begin{flushleft}\textit{#1}\\
\texttt{#2}\end{flushleft}}

\def\myclassification#1{\small\noindent
Keywords : $I^K$-convergence, $I^{K^*}$-convergence, AP$(I,K)$-condition, $P$-ideals,  $I^K$-limit points.
       #1\vspace{0.5cm}\\
			AMS subject classification$(2010)$: Primary: $54$A$20$; Secondary: 40A05; 40A35}
\usepackage{graphicx}% Include figure files
\usepackage{amsmath}

\begin{document}
\mytitle{A Note on $I^K$ and $I^{K^*}$-convergence in topological spaces}

\myname{$Amar~ Kumar~ Banerjee^{\dag}$\footnote{akbanerjee@math.buruniv.ac.in, akbanerjee1971@gmail.com} and  $Mahendranath~ Paul^{\dag}$\footnote{mahendrabktpp@gmail.com}}
\myplace{$\dag$Department of Mathematics, The University of Burdwan, Purba Burdwan$-$713104, India.} {}
\begin{abstract}
In this paper we have studied some important topological properties and characterization of $I^K$-convergence of functions which is a common generalization of $I^*$-convergence of functions. We also introduce the idea of  $I^{K^*}$-convergence and $I^{K}$-limit points of functions.
\end{abstract}
%%%%%%%%%%%%%%%%%%%%%%%%%%%%%%%%%%%%%%%
\myclassification{}
%%%%%%%%%%%%%%%%%%%%%%%%%%%%%%%%%%%%%%%
\section{Introduction and Background}
The aim of this paper is to study the notion of $I^K$ and $I^{K^*}$-convergence of functions which are the common generalization of various type of $I$ and $I^*$-convergence of functions in some restriction. Let us start with brief discussion on two types of ideal convergence.\\
The concept of usual convergence of a real sequence has been extended to statistical convergence by H. Fast\cite{10} and then H. Steinhaus\cite{30} in the year $1951$. Now we recall natural density of a set $K\subset \mathbb{N}$ where $\mathbb N$ denotes the set of natural numbers. Let $K_n$ denote the set $\{k \in K:k\leq n$\} and $|K_n|$ stands for the cardinality of $K_n$.The natural density of $K$ is defined by
$$d(K)=\displaystyle{\lim_{n}}\frac{|K_n|}{n}$$
if the limit exits. A real sequence $\{x_n\}$ is said to be statistically convergent to $l$ if for every $\epsilon>0$ the set $K(\epsilon)=\{k\in\mathbb N:|x_k-l|\geq \epsilon\}$ has natural density zero\cite{10,12,30}. Ordinary convergence always implies statistical convergence\cite{24,28,29}. Later it was developed by many authors and after long $50$ years, the concept of statistical convergence has been extended to $I$ and $I^*$-convergence which depends on the structure of ideals of subsets of the natural numbers by P.Kostyrko et al\cite{18,19,20}.
The concept of $I^*$-convergence which is closely related to that of $I$-convergence and which arises from a particular result on statistical convergence of real sequence was introduced  by P.Kostyrko et al. The result is as follows:\\
A real sequence $\{x_n\}$ is statistically convergent to $\xi$ if and only if there exist a set $M=\{m_1<m_2<m_3<...<m_k<...\}$ such that $d(M)=1$ and $\displaystyle{\lim_{k}}x_{m_k}=\xi$. \cite{18,19}\\
If $I$ is an admissible ideal, $I^*$-convergence implies  $I$-convergence. But converse may not be true. Moreover a statistical convergent sequence and $I$ and $I^*$- convergent sequence need not even be bounded\cite{17,29}. $I$ and $I^*$-convergence coincide for every admissible ideal $I$ if the space is discrete or if $I$ satisfies AP($I$,Fin)-condition.\cite{7,18}. B.K.Lahiri and Pratulananda Das in the year $2005$, extended the concept of $I$ and $I^*$-convergence in a topological space and they observed that the basic properties are preserved also in a topological space\cite{17}. Later many works on $I$-convergence were done in topological spaces\cite{2,3,4,4.2,4.3,4.4}.\\
In the year 2010, M. Macaj and M. Sleziak\cite{25} defined $I^K$-convergence and shew that this type of convergence is a common generalization for all types of $I$ and $I^*$-convergence  we have mentioned so far. They also gave the condition AP$(I,K)$ modifying condition AP from \cite{7,19}. Later in the year 2014, $I^K$-Cauchy and $I^K$-Cauchy net have been studied in \cite{9,26}.\\
In this paper we have studied further some basic properties of $I^K$-convergence of functions in topological spaces which were not studied before. Also we have defined the notion of $I^{K^*}$-convergence and have found out the relation between $I,I^*,K^*,I^{K^*}$ and $I^K$-convergence of functions. While studying the convergence of functions, several closely related notions occur quite naturally such as limit points, cluster points etc. In the last section we have introduced $I^K$-limit points and examined some important topological properties like characterization of compactness in terms of $I^K$-limit points.
\section{Basic Definition and Notation}
\begin{definition}
Let $S$ be a non-void set then a family of sets $I\subset 2^S$ is said to be an ideal if 
\item (i) $A,B\in I \Rightarrow A\cup B\in I$
\item(ii) $A\in I, B\subset A \Rightarrow B\in I$
\end{definition}
$I$ is called nontrivial ideal if $S\notin I$ and $ I\neq \{\phi\} $. In view of condition (ii) $\phi\in I $ i.e. an ideal is a non-void system of sets $I$ hereditary with respect to additive and inclusion. If $I \subsetneq 2^S$ we say that $I$ is proper ideal on $S$. Several examples of non-trivial ideals are seen in \cite{19}.
A nontrivial ideal $I$ is called admissible if it contains all the singleton of $\mathbb N$. A nontrivial ideal $I$ is called non-admissible if it is not admissible. An example of an admissible ideal on a set $S$ is the ideal of all finite subsets of $S$ which we shall denote by Fin$(S)$. If $S=\mathbb N$ then we write Fin instead of Fin$(\mathbb N)$ for short. 
\begin{example}
Let $I$ be the class of all $A\subset \mathbb N$ with $d(A)=0$. Then $I$ is an admissible ideal of $\mathbb{N}$, since singleton sets has density zero. For any proper subset $M\subset \mathbb N$, $I=2^M$ is an non-admissible ideal of $\mathbb{N}$.
\end{example}
\begin{note}
The dual notion to the ideal is the notion of the filter i.e. a filter on $S$ is non-void system of subsets of $S$, which is closed under finite intersection and super sets.
If $I$ is a non-trivial ideal on $X$ then $F=F(I)=\{A\subset X:X\setminus A \in I \}$ is clearly a filter on $X$ and conversely. $F(I)$ is called associated filter with respect to ideal $I$.
\end{note}
Now we will give the definition of $I$-convergence using function instead of sequence.
\begin{definition}\cite{25}
Let $I$ be an ideal on a non-void set $S$ and $X$ be a topological space. A function $f:S\rightarrow X$ is said to be $I$-convergent to $x\in X$, written as $I$-$\lim f=x$ if
$$f^{-1}(U)=\{s\in S:f(s)\in U\}\in F(I)$$ 
for every neighborhood $U$ of the point $x$. i.e. if $f^{-1}(X \setminus U)=\{s\in S:f(s)\notin U\}\in I$ for every neighborhood $U$ of $x$.
\end{definition}
If $S=\mathbb N$ we obtain the usual definition of $I$-convergence of sequence.
\begin{definition}\cite{25}
Let $I$ be an ideal on a set $S$ and let $f:S\rightarrow X$ be a function to a topological space $X$. The function f is called $I^*$-convergent to the point $x$ of $X$ if there exists a set $M\in F(I)$ such that the function $g:S\rightarrow X$ defined by
\[g(s)=\left\{\begin {array}{ll}
        f(s) & \mbox{if $s\in M$} \\
		x & \mbox{if $s\notin M$}
		\end{array}
		\right. \]
is Fin$(S)$-convergent to $x$. 
\end{definition} 
If $f$ is $I^*$-convergent to $x$, then we write $I^*$-$\lim f=x$. The usual notion of $I^*$-convergence of sequence is a special case when $S=\mathbb N$. $I^K$-convergence as a common generalization of all types of $I^*$-convergence of sequences or functions from $S$ into $X$. Here we will work with functions from a non-void arbitrary set $S$ to a topological space $X$. One of the reasons is that using functions sometimes helps to simplify notation.
\begin{definition}\cite{25}
Let $K$ and $I$ be an ideal on a non-void set $S$, $X$ be a topological space and let $x$ be an element of $X$. A function $f:S\rightarrow X$  is called $I^K$-convergent to the point $x$ if there exists a set $M\in F(I)$ such that the function $g:S\rightarrow X$ given by
\[g(s)=\left\{\begin {array}{ll}
        f(s) & \mbox{if $s\in M$} \\
		x & \mbox{if $s\notin M$}
		\end{array}
		\right. \]
is $K$-convergent to $x$. 
\end{definition} 
If $f$ is $I^K$-convergent to $x$, then we write $I^K$-$\lim f=x$.
As usual, notion of $I^K$-convergence of sequence is a special case for $S=\mathbb N$. Similarly as for $I$-convergence of sequences. We write $I^K$-$\lim x_n=x$.
\begin{lemma}\cite{25}\label{0}
If $I$ and $K$ are ideals on a set $S$ and $f:S\rightarrow X$ is a function such that $K$-$\lim f=x$, then $I^K$-$\lim f=x$.
\end{lemma}
\begin{theorem}\cite{25}
Let $I$,$K$ be ideals on a set $S$, $X$ be a topological space and let $f$ be a function from $S$ to $X$ then $I^K$-$\lim f=x\Rightarrow I$-$\lim f=x$ if and only if $K\subset I$.
\end{theorem}
\begin{proposition}\cite{25}
Let $I,I_1,I_2,K,K_1$ and $K_2$ be ideals on a set $S$ such that $I_1\subset I_2$ and $K_1\subset K_2$ and let $X$ be a topological space. Then for any function $f:S\rightarrow X$, we have
$I_1^K$-$\lim f=x ~~~ \Rightarrow ~~~ I_2^K$-$\lim f=x$ and
$I^{K_1}$-$\lim f=x ~~~ \Rightarrow ~~~ I^{K_2}$-$\lim f=x$.
\end{proposition}
\section{Basic Properties of $I^K$-Convergence in Topological Spaces}
Throughout the paper $X$ stands for a topological space $(X,\tau)$ and $I$, $K$ are non-trivial ideals of a non empty set $S$ unless otherwise stated. First we introduce a construction regarding double ideal. For any two ideals  $I,K$ on a non-void set $S$ we have the ideal $$I\vee K=\{A\cup B:A\in I,B\in K\}$$
which is the smallest ideal containing both $I$ and $K$ on $S$ i.e. $I,K\subseteq I\vee K$. It is clear that if $I\vee K$ is non-trivial and $I$ and $K$ are both proper subset of $I\vee K$ then $I$ and $K$ both are non-trivial. But converse part may not be true. To support this following examples are given.
\begin{example}
Consider the two sets $N_1=\{4n:n\in \mathbb{N}\}$ and $N_2=\{4n-1:n\in \mathbb{N}\}$ now it is clear that $2^{N_1}$, $2^{N_2}$ and $2^{N_1}\vee 2^{N_2}$ all are non-trivial ideal on $\mathbb{N}$.
\end{example}
\begin{example}
Now let  $N_1$ be set of all odd integers and $N_2$ be set of all even integers. Then it is clear that $I=2^{N_1}$, $K=2^{N_2}$ both are non-trivial ideals on $\mathbb{N}$ but $I\vee K$ is a trivial ideal on $\mathbb{N}$.
\end{example}
If $I\vee K$ is a non-trivial on $X$ then the dual filter is $F(I\vee K)=\{G\cap H:G\in F(I),H\in F(K)\}$.
\begin{theorem}
Let $I\vee K$ is non-trivial on set $S$. If $X$ is Hausdorff and a function $f:S\rightarrow X$ is $I^K$-convergent then $f$ has a unique $I^K$-limit.
\end{theorem}
\begin{proof}
If possible let us consider that the function $f$ has two distinct $I^K$-limits say $x$ and $y$. Since $X$ is Hausdorff then there exists $U,V\in \tau$ such that $x\in U$ and $y\in V$ and $U\cap V= \phi$. Since $f$ has $I^K$-limit $x$, so from the definition of $I^K$-limit, there exists a set $A_1\in F(I)$ such that the function $g:S \rightarrow X$ given by  
\[g(s)=\left\{\begin {array}{ll}
        f(s) & \mbox{if $s\in A_1$} \\
		x & \mbox{if $s\notin A_1$}
		\end{array}
		\right. \]
is $K$-convergent to $x$. So, $g^{-1}(U)=\{s\in S:g(s)\in U\}=\{s\in A_1:g(s)\in U\}\cup \{s\in S\setminus A_1:g(s)\in U\}=(S\setminus A_1)\cup f^{-1}(U)=S \setminus (A_1\setminus f^{-1}(U))\in F(K)$
i.e. $A_1\setminus f^{-1}(U)\in K$ or $A_1\setminus B_1\in K$ where $B_1= f^{-1}(U)$. Similarly, $f$ has $I^K$-limit $y$ so there exists a set $A_2\in F(I)$ s.t. $A_2\setminus f^{-1}(V)\in K$ or $A_2\setminus B_2\in K$ where $B_2=f^{-1}(V)$.
So,
\begin{equation}\label{1}
(A_1\setminus B_1)\cup (A_2\setminus B_2)\in K
\end{equation}
Now let $x\in (A_1\cap A_2)\cap(B_1\cap B_2)^c = (A_1\cap A_2)\cap (B_1^c\cup B_2^c)=((A_1\cap A_2)\cap B_1^c)\cup ((A_1\cap A_2)\cap B_2^c)$ i.e. either $x\in (A_1\cap A_2)\cap B_1^c\subset A_1\cap B_1^c$ or $x\in ((A_1\cap A_2)\cap B_2^c)\subset A_2\cap B_2^c $ i.e. $x\in (A_1\cap B_1^c)\cup(A_2\cap B_2^c)$. So, $(A_1\cap A_2)\cap(B_1\cap B_2)^c\subset (A_1\cap B^c_1)\cup(A_2\cap B^c_2)\in K$
(from the equation (\ref{1})).
Thus $(A_1\cap A_2)\cap(B_1\cap B_2)^c\in K$ i.e. $(A_1\cap A_2)\setminus (f^{-1}(U)\cap f^{-1}(V))\in K$ i.e.
$(A_1\cup A_2)\setminus (f^{-1}(U\cup V))\in K$. Since $U\cap V= \phi$, then $f^{-1}(U\cap V)=\phi$ so $A_1\cap A_2\in K$ i.e.
\begin{equation}\label{6}
S\setminus (A_1\cap A_2)\in F(K)
\end{equation} 
Since $A_1, A_2\in F(I)$,
\begin{equation}\label{7}
A_1\cap A_2\in F(I)
\end{equation} 
Since $I\vee K$ is non-trivial so the dual filter is $F(I\vee K)=\{G\cap H:G\in F(I),H\in F(K)\}$.
Now using this from \ref{6} and \ref{7} we get $\phi\in F(I\vee K)$, which is a contradiction. Hence the $I^K$-limit is unique.
\end{proof}
\begin{theorem}
If $I$ and $K$ be two admissible ideal and if there exists an injective function $f:S\rightarrow E\subset X$ which is $I^K$-convergent to $x_0\in X$ then $x_0$ is a limit point of $E$
\end{theorem}
\begin{proof}
The function $f$ has $I^K$-limit $x_0$, so $I^K$-limit there exists a set $M\in F(I)$ such that the function $g:S \rightarrow X$ given by  
\[g(s)=\left\{\begin {array}{ll}
        f(s) & \mbox{if $s\in M$} \\
		x_0 & \mbox{if $s\notin M$}
		\end{array}
		\right. \]
is $K$-convergent to $x_0$. Let $U$ be an arbitrary open set containing $x_0$. Then $g^{-1}(U)=\{s:g(s)\in U\}\in F(K)$. So $\{s:g(s)\in U\}\notin K$ i.e. $\{s:g(s)\in U\}$ is an infinite set, as $K$ is an admissible ideal. Choose $k_0\in \{s:g(s)\in U\}$ such that $g(k_0)\neq x_0$ then $g(k_0)\in U\cap (E\setminus \{x_0\})$. Thus $x_0$ is a limit point of $E$.
\end{proof}
\begin{theorem}
A Continuous function $h:X\rightarrow X$ preserves $I^K$-convergence.
\end{theorem}
\begin{proof}
Let the function $f$ has $I^K$-limit $x$, so there exists a set $M\subset S\in F(I)$ such that the function $g:S \rightarrow X$ given by  
\[g(s)=\left\{\begin {array}{ll}
        f(s) & \mbox{if $s\in M$} \\
		x & \mbox{if $s\notin M$}
		\end{array}
		\right. \]
is $K$-convergent to $x$. Let $U$ be an arbitrary open set containing $x$. Then $g^{-1}(U)=S\setminus (M\setminus f^{-1}(U))\in F(K)$ i.e. $M\setminus f^{-1}(U)\in K$. So to prove the theorem we have to show that $I^K$-$\lim h(f(x))=h(x)$ i.e. it suffices to show that the function $g_1:S \rightarrow X$ given by
\[g_1(s)=\left\{\begin {array}{ll}
        (h\circ f)(s) & \mbox{if $s\in M$} \\
		h(x) & \mbox{if $s\notin M$}
		\end{array}
		\right. \]
is $K$-convergent to $h(x)$. Let $V$ be an open set containing $h(x)$. Since $h$ is continuous so there exists an open set $U$ containing $x$ such that $h(U)\subset V$. Clearly
$\{x: h(f(x))\notin V\}\subset \{x:f(x)\notin U\}$ which implies that $\{x:f(x)\in U\}\subset \{x:h\circ f(x)\in V\}$ i.e. $f^{-1}(U)\subset (h\circ f)^{-1}(V)$. So $M\setminus (h\circ f)^{-1}(V)\subset M\setminus f^{-1}(U)$. Then  $M\setminus (h\circ f)^{-1}(V)\in K$ as $M\setminus f^{-1}(U)\in K$. So its complement $g_1^{-1}(V)\in F(K)$, as required. Hence $I^K$-$\lim (h\circ f)(x)=h(x)$.
\end{proof}
\begin{theorem}
If $X$ is a discrete space then $I$-convergence implies $I^K$-convergence, where $I$ and $K$ are two admissible ideals.
\end{theorem}
\begin{proof}
Let $f:S \rightarrow X$ be a function such that $I$-$\lim f=x_0$. Since $X$ is a discrete space so it has no limit point then $U=\{x_0\}$ is open. Thus we have $f^{-1}(X\setminus U)=\{s\in S: f(s)\notin U\}\in I$. Let the set $M= f^{-1}(U)=\{s\in S: f(s)\in U\}\in F(I)$. Thus there exists a set $M=\{s:f(s)\in U\}=\{s:f(s)=x_0\}\in F(I)$ such that the function $g:S\rightarrow X$ defined by
\[g(s)=\left\{\begin {array}{ll}
        f(s) & \mbox{if $s\in M$} \\
		x_0 & \mbox{if $s\notin M$}
		\end{array}
		\right. \]
is $K$-convergent to $x_0$, since for any open set $U$ containing $x_0$  the set $g^{-1}(U)=S \in F(K)$. Hence $I^K$-$\lim f= x_0$
\end{proof}
\begin{note}
Converse of above theorem may not be true. Let $I$ and $K$ be two ideals on a set $S$. Consider a set $A\in K\setminus I$. Let $y_0\in X\setminus \{x_0\}$ be a fixed element and define a function $f:S\rightarrow X$ by
\[f(s)=\left\{\begin {array}{ll}
        x_0 & \mbox{if $s\in S\setminus A$} \\
		y_0 & \mbox{otherwise}
		\end{array}
		\right. \]
Now if $V$ is any open set containing $x_0$ then $f^{-1}(V)=S\setminus A$ if $y_0\notin V$ and $f^{-1}(V)=S$ if $y_0\in V$. So in both case $f^{-1}(V)\in F(K)$. Hence $K$-$\lim f=x_0$ then by lemma (2.1) we get $I^K$-$\lim f=x_0$. But $U=\{x_0\}$ is an open set containing $x_0$ since $X$ is a discrete space and $f^{-1}(X\setminus U) = A \notin I$. Hence $f$ is not $I$-convergent to $x_0$.
\end{note}
\begin{theorem}
Let $(X,\tau)$ be a topological space and let $f:S\rightarrow X$ be a function, where $S$ is a non-empty set, such that $x\in X$ is an $I^K$-limit of the function $f$, for some non-trivial ideals $I$ and $K$ of $S$. Then there exists a filter $F$ on $X$ such that $x$ is also a limit of the filter $F$.
\end{theorem}
\begin{proof}
Let $I$ $\&$ $K$ be two non-trivial ideals on non-empty set $S$. Also let $x$ is $I^K$-limit of the function $f:S\rightarrow X$. Then from the definition of $I^K$-convergence then there exists a set $M_1\in F(I)$ such that the function $g:S \rightarrow X$ given by  
\[g(s)=\left\{\begin {array}{ll}
        f(s) & \mbox{if $s\in M_1$} \\
		x & \mbox{if $s\notin M_1$}
		\end{array}
		\right. \]
is $K$-convergent to $x$.
So for every open set $U$ containing $x$, the set
\begin{equation}\label{4}
M=g^{-1}(U)=\{s\in S:g(s)\in U\}\in F(K)
\end{equation}
Let us construct for each $M\in F(K)$ the set $A_M=\{g(n):n\in M\}$ and $\mathcal{B}=\{A_M:M\in F(K)\}$. Then the family $\mathcal{B}$ forms a filter base on $X$. In fact, (i)We observe that each $A_M$ is non-empty. Since $M$ is non-empty so $\mathcal{B}$ is non-empty. (ii)Since $F(K)$ is filter, $\phi \notin F(K)$ and so $A_M\neq \phi$ for all $M\in F(K)$ and $\phi \notin \mathcal{B}$. (iii) Let us take any two members $A_M,A_R\in$ where $M,R\in F(K)$. $M\cap R\in F(K)$ since $F(K)$ is filter on $S$. So $A_{M\cap R}\in \mathcal{B}$. Also $A_{M\cap R}\subset A_M\cap A_R$. So $\mathcal{B}$ is a filter base. Let $F$ be the filter generated by this filter base. Now we will show that $x$ be the limit of filter $F$. Let $V$ be any open set of $x$. Then from (\ref{4}) the set $ M=\{s\in S:g(s)\in V\}\in F(K)$. So by our construction of $A_M$, we get $A_M=\{g(n):n\in M\}\subset V$. Since $A_M\in \mathcal{B}$ we get $V\in F$. So we conclude that $V\in F$ for all open set $V$ of $x$. Hence $x$ becomes limit of the filter $F$.
\end{proof}
\begin{theorem}
Let $(X,\tau)$ be a topological space and $x\in X$. Then for every function $f:S\rightarrow X$  there exists a filter $F$ on $X$ such that if $x$ is limit of filter $F$ then $x$ is also $I^K$-limit of the function $f$.
\end{theorem}
\begin{proof}
Let $f:S\rightarrow X$ be a function and $I$, $K$ be two non-trivial ideals of $S$. For each $M\in F(K)$ let $A_M=\{f(n):n\in M\}$ and $\mathcal{B}=\{A_M:M\in F(K)\}$. Then the family $\mathcal{B}$ forms a filter base on $X$. Let $F$ be the filter generated by this filter base. Let $x$ be the limit of filter $F$. Then $\eta_x\subset F$ where $\eta_x$ is the neighborhood filter  of the point $x$. Let $U\in \eta_x$ be arbitrary. Then $U\in F$ and so $A_M\subset U$ for some $M\in F(K)$. This implies that $M\subset \{n\in S:f(n)\in U\}$  which further implies that $\{n\in S:f(n)\in U\}\in F(K)$ since $M\in F(K)$. Now $U$ is arbitrary so the function $f$ is $K$-convergent to $x$. Hence from the lemma (\ref{0}) we get $f$ is $I^K$-convergent to $x$. 
\end{proof}
\section{$I^{K^*}$-Convergence in Topological Spaces}
$I^{K^*}$-convergence is also a common generalization of all types of $I^*$ and  $K^*$-convergence. It is interesting to find the relation between $I, I^*,K^*,I^{K^*}$ and $I^K$-convergence.
\begin{definition}
Let $X$ be a topological space and $x\in X$ and let $I$, $K$ be two ideals on a non-void set $S$. A function $f:S\rightarrow X$  is called $I^{K^*}$-convergent to the point $x$ if there exists a set $M\in F(I)$ and $M_1 \in F(K)$ such that the function $g:S\rightarrow X$ given by
\[g(s)=\left\{\begin {array}{ll}
        f(s) & \mbox{if $s\in M\cap M_1$} \\
		x & \mbox{if $s\notin M\cap M_1$}
		\end{array}
		\right. \]
is Fin$(S)$-convergent to $x$. 
\end{definition}
If $f$ is $I^{K^*}$-convergent to $x$ then we write $I^{K^*}$-$\lim f=x$.
\begin{note} It follows from the definition that 
$f$ is $I^{K^*}$-convergent to $x$ if and only if there exist a set $M\in F(I)$ such that the function $g:S\rightarrow X$ given by
\[g(s)=\left\{\begin {array}{ll}
        f(s) & \mbox{if $s\in M$} \\
		x & \mbox{if $s\notin M$}
		\end{array}
		\right. \]
is $K^*$-convergent to $x$. 
\end{note}
\begin{lemma}
If $I$ and $K$ are two ideals on a set $S$ and if $f:S\rightarrow X$ is a function such that $K^*$-$\lim f=x$ then $I^{K^*}$-$\lim f=x$.
\end{lemma}
\begin{proof}
Follows from the lemma \ref{0}.
\end{proof}
\begin{lemma}\label{2}
If $I$ and $K$ be two admissible ideals on a set $S$ and $f:S\rightarrow X$ is a function such that $I^{K^*}$-$\lim f=x$ then $I^K$-$\lim f=x$.
\end{lemma}
\begin{proof}
The proof follows from the note $(4.1)$ and since $K^*$-convergence implies $K$-convergence of the function $g$.
\end{proof}
\begin{theorem}
If $X$ is a discrete space then $I^K$ and $I^{K^*}$-convergence coincide for every admissible ideal $I$ and $K$.
\end{theorem}
\begin{proof}
Let $X$ be a discrete topological space then it has no limit point and $x\in X$. Let $I$ and $K$ be two admissible ideals on a set $S$ and $f:S\rightarrow X$ is a function such that $I^K$-$\lim f=x$. Because of previous lemma (\ref{2}) we have only to show that $I^{K^*}$-$\lim f=x$. Now from the definition of $I^K$-convergence there exists a set $M\in F(I)$ such that the function $g:S \rightarrow X$ defined by
\[g(s)=\left\{\begin {array}{ll}
        f(s) & \mbox{if $s\in M$} \\
		x & \mbox{if $s\notin M$}
		\end{array}
		\right. \]
 is $K$-convergent to $x$ i.e. $K$-$\lim g(x)=x$. Since $X$ has no limit point so $U=\{x\}$ is open. So we have $\{s:g(s)\notin U\}\in K$. Hence the set $M_1=\{s: g(s)\in U\}=\{s:g(s)=x\}\in F(K)$. So there exist $M_1\in F(I)$ such that the function $g_1:S\rightarrow X$ defined by
\[g_1(s)=\left\{\begin {array}{ll}
        f(s) & \mbox{if $s\in M_1$} \\
		x & \mbox{if $s\notin M_1$}
		\end{array}
		\right. \]
is Fin$(S)$-convergent to $x$, since for any open set $U$ containing $x$, $g^{-1}(X\setminus U)=\phi$ is a finite set. 
 Thus $K^*$-$\lim g(x)=x$. So $I^{K^*}$-$\lim f=x$. 
\end{proof}
\begin{theorem}
Let $I$ and $K$ be two admissible ideals on a non-empty set $S$ and let $f:S\rightarrow X$ be a function where $X$ is a topological space. Then $I^{K^*}$-convergence implies $I$-convergence if $K\subseteq I$.
\end{theorem}
\begin{proof}
Suppose that the function $f: S\rightarrow X$ is $I^{K^*}$-convergent to $x\in X$. So there exists sets $M\in F(I)$ and $M_1\in F(K)$ such that the function $g:S\rightarrow X$ given by
\[g(s)=\left\{\begin {array}{ll}
        f(s) & \mbox{if $s\in M\cap M_1$} \\
		x & \mbox{if $s\notin M\cap M_1$}
		\end{array}
		\right. \]
is Fin$(S)$-convergent to $x$ i.e. $g^{-1}(X\setminus U)=\{s\in S:g(s)\notin U\}$ is a finite set for each open set $U$ containing the point $x$. Now the set $C$ (say)$=f^{-1}(X\setminus U)\cap(M\cap M_1)\subset g^{-1}(X\setminus U)$ i.e. $C$ is finite. So $C\in I$. Now,
\begin{equation}\label{3}
f^{-1}(X\setminus U)\subseteq (S\setminus (M\cap M_1))\cup C
\end{equation} 
and $F(K)\subset F(I)$, since $K\subseteq I$ . Therefore $M\cap M_1\in F(I)$. So $S\setminus (M\cap M_1)\in I$. So from (\ref{3}) we get $f^{-1}(X\setminus U)\in I$. Therefore $f$ is $I$-convergent to $x$. i.e.$I$-$\lim f= x$
\end{proof}
\begin{lemma}
If $I$ and $K$ be two admissible ideals on a set $S$ and $f$ be a function from $S$ to $X$, where $X$ be a topological space. Then $I^{K^*}$-convergence implies $K$-convergence if $I\subseteq K$.
\end{lemma}
\begin{proof}
The proof is similar to the proof of Theorem $(4.2)$ and so omitted.
\end{proof}
\begin{theorem}
$I^*$-convergence implies $I^{K^*}$-convergence.
\end{theorem}
\begin{proof}
Let $I$ and $K$ be two ideals on a non-void set $S$ and $f:S\rightarrow X$ be a function such that $f$ is $I^*$-convergence to $x$ of $X$. So $\exists$ a set $M\in F(I)$ such that the function $g:S\rightarrow X$ defined by
\[g(s)=\left\{\begin {array}{ll}
        f(s) & \mbox{if $s\in M$} \\
		x & \mbox{if $s\notin M$}
		\end{array}
		\right. \]
is Fin$(S)$-convergent to $x$.
Since Fin-convergent always implies $K^*$-convergent then the function $g$ is $K^*$-convergent to $x$. and so $f$ is $I^{K^*}$-convergent to $x$ by the Note$(4.1)$.
\end{proof}
\begin{lemma}
$K^*$-convergence implies $I^{K^*}$-convergence.
\end{lemma}
\subsection{Additive Property with $I^K\& I^{K^*}$-Convergence}
We now study the relationship between $I,I^{K^*}\& I^K$-convergence. The following definition is important in this regard.
\begin{definition}\cite{9}
Let $I,K$ be ideals on the non-empty set $S$. We say that $I$ has additive property with respect to $K$ or that the condition AP$(I,K)$ holds if for every sequence of pairwise disjoint sets $A_n\in I$, there exists a sequence $B_n\in I$ such that $A_n\bigtriangleup B_n\in K$ for each $n$ and $\displaystyle \cup_{n\in \mathbb N}B_n\in I$ 
\end{definition}
Another formulation of condition AP$(I,K)$ are given in \cite{25}. Before giving this definition we need to state definition of $K$-pseudo-intersection of a system.
\begin{definition}\cite{25}
Let $K$ be an ideal on a set $S$. We write $A\subset_K B$ whenever $A\setminus B\in K.$ If $A\subset_K B$ and $B\subset_K A$ then we write $A\sim_K B$. Clearly $A\sim_K B \Leftrightarrow A\bigtriangleup B\in K$\\
We say that a set $A$ is $K$-pseudo-intersection of a system $\{A_n: n\in \mathbb N\}$ if $A\subset_K A_n$ holds for each $n\in \mathbb N$
\end{definition}
\begin{definition}\cite{25}
Let $I,K$ be ideals on the set $S$. We say that $I$ has additive property with respect to $K$ or that the condition AP$(I,K)$ holds if any of the equivalent condition of following holds:
\item[(i)] For every sequence $(A_n)_{n\in \mathbb N}$ of sets from  $I$ there is $A\in I$ such that $A_n\subset_K A$ for all $n'$s.
\item[(ii)] Any sequence $(F_n)_{n\in \mathbb N}$ of sets from $F(I)$ has $K$-pseudo-intersection in $F(I)$.
\item[(iii)] For every sequence $(A_n)_{n\in \mathbb N}$ of sets from  $I$ there exists a sequence $(B_n)_{n\in \mathbb N}\in I$ such that $A_j\sim_K B_j$ for $j\in \mathbb N$ and $B=\displaystyle{\cup_{j\in \mathbb N}}B_j\in I$.
\item[(iv)] For every sequence of mutually disjoint sets $(A_n)_{n\in \mathbb N}\in I$ there exists a sequence $(B_n)_{n\in \mathbb N}\in I$ such that $A_j\sim_K B_j$ for $j\in \mathbb N$ and $B=\displaystyle{\cup_{j\in \mathbb N}}B_j\in I$.
\item[(v)] For every non-decreasing sequence $A_1\subseteq A_2\subseteq \cdots \subseteq A_n\cdots $ of sets from $I$ $\exists$  a sequence $(B_n)_{n\in \mathbb N}\in I$ such that $A_j\sim_K B_j$ for $j\in \mathbb N$ and $B=\displaystyle{\cup_{j\in \mathbb N}}B_j\in I$.
\item[(vi)] In the Boolean algebra $2^S/K$ the ideal $I$ corresponds to a $\sigma$-directed subset,i.e. every countable subset has an upper bound.
\end{definition}
In the case $S=\mathbb N$ and $K=$ Fin we get the condition AP from \cite{19} which characterize ideal such that $I^*$-convergence implies $I$-convergence. The condition AP$(I,K)$ is more generalization of condition AP from\cite{7}\cite{19} . Ideals which fulfill the condition AP($I$,Fin) are sometimes called $P$-ideals.(see for examples \cite{1}\cite{11}) \\
In the paper \cite{25} the author showed that $I$-convergence implies $I^K$-convergence if AP$(I,K)$ holds. Here we will introduce a new theorem regarding $I$ and $I^{K^*}$-convergence.
\begin{theorem}
Let $I$ and $K$ be two ideals on a set $S$ and $X$ be a first countable topological space. If the ideal $I$ has the additive property with respect to P-ideal $K$ then $I$-convergence implies $I^{K^*}$-convergence.
\end{theorem}
\begin{proof}
Let $f:S\rightarrow X$ be a function such that $I$-$\lim f=x_0$. Let $\mathcal{B}=\{U_n:n\in \mathbb N\}$ be a countable base for $X$ at the point $x_0$. Now from the definition of $I$-convergence we have $f^{-1}(U_n)\in F(I)$ for each n. Thus there exists $A\in F(I)$ with $A\subset_K f^{-1}(U_n)$ for each $n$ i.e. $A\setminus f^{-1}(U_n)\in K$.
Now it suffices to show that the function the $g:S\rightarrow X$ defined by
\[g(n)=\left\{\begin {array}{ll}
        f(n) & \mbox{if $n\in A$} \\
		x_0 & \mbox{if $n\notin A$}
		\end{array}
		\right. \]
is $K^*$-convergent to $x_0$. For $U_n\in B$, we have $g^{-1}(U_n)=(S\setminus A)\cup f^{-1}(U_n)=S\setminus (A\setminus f^{-1}(U_n))$ and since the set $A\setminus f^{-1}(U_n) \in K$ so $S\setminus (A\setminus f^{-1}(U_n))\in F(K)$ i.e. $g^{-1}(U_n)\in F(K)$. Therefore $g$ is $K$-convergent to $x_0$. Since $K$ is P-ideal so $g$ is also $K^*$-convergent to $x_0$.
\end{proof}
\section{$I^K$-Limit Points}
We modify the definition of $I$-limit points in the following way:
\begin{definition}
Let $f:S\rightarrow X$ be a function and $I$  be non-trivial ideal of $S$. Then $y\in X$ is called an $I$-limit point of $f$ if there exists a set $M\subset S$ such that $M\notin I$ and the function $g:S\rightarrow X$ defined by
\[g(s)=\left\{\begin {array}{ll}
        f(s) & \mbox{if $s\in M$} \\
		y & \mbox{if $s\notin M$}
		\end{array}
		\right. \]
is Fin$(S)$-convergent to $y$.
\end{definition}
\begin{definition}
Let $f:S\rightarrow X$ be a function and $I,K$  be two non-trivial ideals of $S$. Then $y\in X$ is called an $I^K$-limit point of $f$ if there exists a set $M\subset S$ such that $M\notin I,K$ and the function $g:S\rightarrow X$ defined by
\[g(s)=\left\{\begin {array}{ll}
        f(s) & \mbox{if $s\in M$} \\
		y & \mbox{if $s\notin M$}
		\end{array}
		\right. \]
is $K$-convergent to $y$.
\end{definition}
We denote respectively by $I(L_f)$ and $I^K(L_f)$ the collection of all $I$ and $I^K$-limit points of $f$.
\begin{theorem}
If $K$ is an admissible ideal and $K\subset I$ then $I(L_f)\subset I^K(L_f)$
\end{theorem}
\begin{proof}
Let $y\in I(L_f)$. Since $y$ is an $I$-limit point of the function $f:S\rightarrow X$, then there exists a set $M\notin I$ such that and the function $g:S\rightarrow X$ defined by
\[g(s)=\left\{\begin {array}{ll}
        f(s) & \mbox{if $s\in M$} \\
		y & \mbox{if $s\notin M$}
		\end{array}
		\right. \]
is Fin$(S)$-convergent to $y$. So for any open set $U$ containing $x$ the set $\{s:g(s)\notin U\}\in$ Fin. i.e. $\{s:g(s)\notin U\}$ is a finite set. So $\{s:g(s)\notin U\}\in K$, as $K$ is an admissible ideal. Therefore $g$ is $K$-convergent function. Again $M\notin I$ and $K\subset I$ so $M\notin I,K$. Thus $y$ is $I^K$-limit point of $f$ i.e. $y\in I^K(L_f)$. Hence the theorem is proved.
\end{proof}
\begin{note}
If $I$ is an admissible ideal and $I\subset K$ then $K(L_f)\subset I^K(L_f)$
\end{note}
\begin{theorem}
If every function $f:S\rightarrow X$ has an $I^K$-limit point then every infinite set $A$ in $X$ has an $\omega$-accumulation point where cardinality of $S$ is less or equal to cardinality of $A$.
\end{theorem}
\begin{proof}
Let $A$ be an infinite  set. Define an injective function $f: S\rightarrow A\subset X$. Then $f$ has an $I^K$-limit point say $y$. Then $\exists$ a set $M\subset S$ such that $M\notin I,K$ and the function $g:S\rightarrow X$ given by
\[g(s)=\left\{\begin {array}{ll}
        f(s) & \mbox{if $s\in M$} \\
		y & \mbox{if $s\notin M$}
		\end{array}
		\right. \]
is $K$-convergent to $y$. Let $U$ be open set containing $y$ then $g^{-1}(U)=(S\setminus M)\cup f^{-1}(U)=S\setminus (M\setminus f^{-1}(U))\in F(K)$ i.e. $M\setminus f^{-1}(U)\in K$. So $f^{-1}(U)\notin K$.(For if $f^{-1}(U)\in K$ then we get $M\in K$, which is a contradiction.) So $\{s:f(s)\in U\}$ is an infinite set. Consequently $U$ contains infinitely many points of the function $f(s)$ in $X$. So $U$ contains infinitely many elements of $A$. Thus $y$ becomes $\omega$-accumulation point of $A$.
\end{proof}
\begin{theorem}
If $X,\tau$ is a Lindelof space such that every function $f: \mathbb N\rightarrow X$ has an $I^K$-limit point then $(X,\tau)$ is compact. 
\end{theorem}
\begin{proof}
Let $(X,\tau)$ be a Lindelof space such that every $f: N\rightarrow X$ has an $I^K$-limit point. We have to show that any open cover of space $X$ has a finite subcover. Let $\{A_\alpha: \alpha\in \wedge\}$ be an open cover of the space $X$, where $\wedge$ is an index set. Since $(X,\tau)$ is a Lindelof space so this open cover admits a countable sub-cover say $\{A_1,A_2,\cdots,A_n,\cdots\}$. Proceeding inductively let $B_1=A_1$ and for each $m>1$, let $B_m$ be the first member of the sequence of $A'$s which is not covered by $B_1\cup B_2\cup B_3\cup \cdots\cup B_{m-1} $. If this choice becomes impossible at any stage then the sets already selected becomes a required finite sub-cover. Otherwise it is possible to select a point $b_n$ in $B_n$ for each positive integer n such that $b_n\notin B_r, r<n$. \\
Let $f: \mathbb N \rightarrow X$ be a function defined by $f(n)=b_n$. Now let $x$ be an $I^K$-limit point of the function $f$. Then $x\in B_p$ for some $p$. Now from the definition of $I^K$-limit point we get $g^{-1}(B_p)=(\mathbb N\setminus M)\cup f^{-1}(B_p)=\mathbb N\setminus (M\setminus f^{-1}(B_p))\in F(K)$ i.e. $M\setminus f^{-1}(B_p)\in K$. So the set $S=f^{-1}(B_p)=\{n\in \mathbb N:f(x_n)\in B_p\}\notin K$. Hence $S$ must be an infinite subset of $\mathbb N$. So there is some $q>p$ such that $q\in S$ i.e. there exists some $q>p$ such that $f(x_q)\in B_p$ which leads to a contradiction. Thus the result follows.
\end{proof}
%%%%%%%%%%%%%%%%%%%%%%%%%%%%%%%%%%%
\section{Acknowledgment}
The second author is grateful to the University of Burdwan for providing State Fund Fellowship during the preparation of this work.
%%%%%%%%%%%%%%%%%%%%%%%%%%%%%%%%%%%

\end{document}